\documentclass[12pt,envcountsame]{amsart}
\usepackage[utf8]{inputenc}   
\usepackage[russian,english]{babel}
\usepackage{enumerate}

\usepackage{relsize}
\usepackage{exscale}

\makeatletter
\newcommand{\raisemath}[1]{\mathpalette{\raisem@th{#1}}}
\newcommand{\raisem@th}[3]{\raisebox{#1}{$#2#3$}}
\makeatother

\usepackage[colorlinks]{hyperref}
\usepackage{amsmath}
\usepackage{amsthm}
\usepackage{amsfonts}
\usepackage{amssymb}
\usepackage{cite}

\usepackage{lmodern}
\usepackage{fonttable}

\usepackage{mathptmx} 
\usepackage[scaled=0.90]{helvet} 
\usepackage{courier} 
\usepackage[T1]{fontenc}

\usepackage[dvips]{graphicx}
\usepackage{graphicx}
\DeclareGraphicsExtensions{.eps}

\newtheorem{theorem}{Theorem}[section]
\newtheorem{utv*}{Proposition}
\newtheorem{hyp*}{Conjecture}
\newtheorem*{example*}{Example}
\newtheorem{lemma}[theorem]{Lemma}
\newtheorem{corollary}[theorem]{Corollary}

\newtheorem*{th*}{Theorem}

\theoremstyle{definition}

\newtheorem{defin}[theorem]{Definition}
\newtheorem{example}[theorem]{Example}

\def\sli{\sum\limits}
\def\ili{\int\limits}

\def\R{\mathbb{R}}

\def\ep{\varepsilon}
\def\vf{\varphi}

\def\B{\mathbb{B}}

\def\clos{\textup{clos}}
\def\NN{\mathbb N}
\def\R{\mathbb R}
\def\EE{\mathcal{E}}

\renewcommand{\L}{\mathcal{L}}

\newcommand{\dist}{\operatorname{dist}}

\newcommand{\PP}{\mathcal{P}}

\DeclareTextFontCommand{\textcyr}{\cyr}

{\end{list}}

\newcounter{vremennyj}

\def\o{\overline}
\def\u{\underline}
\def\H{\mathcal{H}}

\renewcommand\S{\mathbb{S}}

\title[Optimal discrete measures]{Optimal discrete measures for Riesz potentials}

\date{\today}
\begin{document}
\author{S. V. Borodachov}
\address{Department of Mathematics, Towson University}
\email{sborodachov@towson.edu}
\author{D. P. Hardin$^\dagger$}
\address{Center for Constructive Approximation, Department of Mathematics, Vanderbilt University}
\email{doug.hardin@vanderbilt.edu}
\author{A. Reznikov$^\dagger$}
\address{Center for Constructive Approximation, Department of Mathematics, Vanderbilt University}
\email{aleksandr.b.reznikov@vanderbilt.edu}
\author{E. B. Saff$^\dagger$}
\address{Center for Constructive Approximation, Department of Mathematics, Vanderbilt University}
\email{edward.b.saff@vanderbilt.edu}

\thanks{$^\dagger$ This research was supported, in part,
by the U. S. National Science Foundation under the grant DMS-1412428 and DMS-1516400.
}

\maketitle

\date{}

\begin{abstract}
For weighted Riesz potentials of the form $K(x,y)=w(x,y)/|x-y|^s$, we investigate $N$-point configurations $x_1,x_2, \ldots, x_N$ on a $d$-dimensional
compact subset $A$ of $\mathbb{R}^p$ for which the minimum of $\sum_{j=1}^NK(x,x_j)$ on $A$ is maximal.   Such quantities are called
$N$-point Riesz $s$-polarization (or Chebyshev) constants. 
 For $s\geqslant d$, we obtain the
dominant term  as $N\to \infty$  of such constants for a class of $d$-rectifiable  subsets of
$\mathbb{R}^p$.  This class includes
 compact subsets of $d$-dimensional $C^1$ manifolds whose boundary relative to the manifold has $d$-dimensional Hausdorff measure zero, as well as finite unions of such sets when their pairwise intersections have measure zero.  We also explicitly determine the weak-star limit distribution of asymptotically optimal $N$-point
 configurations for weighted $s$-polarization as $N\to \infty$.
\end{abstract}

\vspace{4mm}


\vspace{2mm}


\vspace{2mm}

\normalsize

\section{Introduction}\label{S1}
For a compact set $A\subset \R^p,$ two classical geometric problems are that of best-packing and best-covering by
an $N$-\emph{point multi-set} (or $N$-\emph{point configuration})  $\omega_N=\{x_1,  \ldots, x_N \}\subset A$; i.e., a set of
points with possible repetitions and cardinality  $\#\omega_N=N.$  The former problem is to determine the largest possible
separation distance that can be attained by $N$ points of $A$:
$$\delta_N(A):=\max_{\omega_N \subset A}\min_{i \neq j}|x_i-x_j|,
$$
while the latter is to find the smallest radius so that the union of $N$ closed balls of this radius centered at points of $A$ covers $A$:
$$\rho_N(A):=\min_{\omega_N\subset A}\max_{y \in A}\min_{x \in \omega_N}|x -y |.
$$
These two problems are referred to by some authors as being `somewhat dual' (cf. \cite{MR1662447}). They are, in fact, limiting cases of certain minimal energy and maximal Chebyshev (polarization) problems for strongly repulsive kernels as we now describe.

Given a lower semi-continuous kernel $K(x,y):A\times A \to (-\infty, \infty]$ and an $N$-point configuration $\omega_N$ as above, its
$K$-energy is


$$E_K(\omega_N):=\sli_{1\leqslant i\not=j \leqslant N}K(x_i, x_j),
$$
and we denote by $\EE_K(A; N)$ the minimal $K$-energy over all such $N$-point configurations:
$$\EE_K(A; N):=\min_{\omega_N \subset A} \{E_K(\omega_N)\}.
$$
Determining $N$-point configurations $\omega_N^*$ such that $E_K(\omega_N^*)=\EE_K(A;N)$; i.e., finding $N$-\emph{point equilibrium configurations}, is in general a difficult problem having classical roots (e.g. the Thomson problem \cite{thomson1904xxiv} for electrons on the sphere). For strongly repulsive kernels
 $K$, minimal discrete energy problems resemble best-packing ones.\

 The less studied notion of maximal polarization (or  maximal Chebyshev constant) is the following. Let
$$
U_K(y; \omega_N):=\sli_{i=1}^N K(y, x_i)
$$
and consider its minimum:
$$
P_K(A; \omega_N):=\min_{y\in A}U_K(y; \omega_N).
$$
Then the {\it $N$-th $K$-polarization (or Chebyshev) constant of $A$} is defined by
\begin{equation}\label{optpol}
\PP_K(A; N):=\max_{\omega_N \subset A} P_K(A; \omega_N),
\end{equation}
and we say that $\omega_N^*$ is \emph{an optimal} (or \emph{maximal}) $K$-\emph{polarization configuration} whenever $P_K(A; \omega_N^*)=\PP_K(A; N).$
For example, if $A$ is the interval $[-1,1]$ and $K$ is  the logarithmic kernel, $K_{\log}(x,y):=-\log|x-y|$, then the optimal $N$-point log-polarization configuration consists of the zeros of the Chebyshev polynomial $\cos(N \arccos x)$. Furthermore, for an arbitrary compact subset $A$ of the plane, the limiting behavior (as $N\to\infty)$  of $\PP_{\log}(A; N)$ determines the logarithmic capacity of   $A$ (see, e.g. \cite{MR1485778}).\

We  remark that from an applications prospective, the maximal polarization problem, say on a compact surface (or volume), can be viewed as the problem of determining the smallest number of sources (injectors) of a substance together with their optimal locations that can provide a required dosage of the substance to every point of the surface (volume). Such problems arise, for example, in the implantation of radioactive seeds for the treatment of a tumor.\

 The precise connections of the minimal energy and maximal polarization problems to best-packing and best-covering are as follows. Let
$$K_s(x,y):=\frac{1}{|x-y|^s},\,\,s>0, $$
denote the Riesz $s$-kernel. Then for   $N$ fixed,
$$\lim_{s\to \infty}\left[\EE_{K_s}(A; N)\right]^{1/s}=\frac{1}{\delta_N(A)},\quad N\ge 2,
$$
and
$$
\lim_{s\to \infty} \left[\PP_{K_s}(A; N)\right]^{1/s}=\frac{1}{\rho_N(A)}, \quad N\ge 1.
$$
 Moreover, every limit configuration (as $s \to \infty$) of optimal $N$-point configurations for the discrete $s$-energy and $s$-polarization problems is an $N$-point
 best-packing, respectively, best-covering configuration for $A$ (see \cite{Borodachov2007},\cite{Borodachov2015}).\

 While Riesz equilibrium configurations have been much studied (see e.g. \cite{Dahlberg1978}, \cite{MR1485778}, \cite{Landkof1972}, \cite{Kuijlaars1998}, \cite {Hardin2005},\cite{Borodachov2015}), polarization problems are somewhat more
 difficult to tackle. For example, if $A$ is the unit circle $\S^1$ and $s>0$, then it is fairly straightforward (using a convexity
 argument) to show that minimal $N$-point Riesz $s$-equilibrium configurations are given by $N$ equally spaced points. However, the analogous
 problem for $N$-point maximal polarization configurations (which everyone would guess has the same solution) was a
 conjecture of Ambrus, Ball, and Erd\'{e}lyi \cite{Ambrus2013} for which only partial results \cite{Ambrus2009}, \cite{Ambrus2013}, \cite{Erdelyi2013} existed until a rather subtle general proof
 was presented in \cite{Hardin2013}. Similarly, when $A=\S^2$ (the unit sphere in $\R^3$), $s>0$, and $N=4$, the vertices of the inscribed tetrahedron are
 optimal both for minimal energy and maximal polarization, but the proof of the latter is far more difficult than that of the former (see \cite{Su2014}).\

    The goal of the present paper is to study the asymptotic behavior (as $N\to \infty$) of maximal $N$-point Riesz $s$-polarization
configurations on manifolds embedded in $\R^p$ for the so-called `hypersingular (or nonintegrable) case' when $s> \dim(A)$, where $\dim(A)$ denotes the Hausdorff dimension of $A$.  Our results can be considered
as dual to those on minimal energy that appeared in this journal \cite{Borodachov2008}. While some arguments
developed for those minimal energy problems can be adapted to our purpose, the investigation of polarization configurations requires
some novel techniques, as foreshadowed by the examples mentioned above. For instance, while  minimal energy has a simple monotonicity
property: $A \subset B\, \Rightarrow \,\EE_K(B; N) \leq \EE_K(A; N)$,
no such analogous property holds for polarization.

The notion of polarization for potentials was likely first introduced by Ohtsuka (see, e.g., \cite{Ohtsuka1967}) who explored (for very general kernels) their relationship to various definitions of capacity that arise in electrostatics. In particular, he showed that for any compact set $A\subset \R^p$ the following limit, called {\it the Chebyshev constant} of $A$, always exists as an extended real number:
\begin {equation}\label {s<d}
T_K(A):=\lim\limits_{N\to\infty}{\frac {\mathcal P_K(A;N)}{N}},
\end {equation}
and, moreover, is given by the continuous analogue of polarization:
\begin {equation}\label {mu^s}
T_K(A)=\sup\limits_{\mu\in \mathfrak M(A)}\inf\limits_{y\in A}U_K^{\mu}(y),
\end {equation}
where $\mathfrak M(A)$ is the set of all Borel probability measures supported on $A$, and
$$
U_K^{\mu}(y):=\int_A K(x,y) \textup{d}\mu(x).
$$
Ohtsuka further showed that $T_K(A)$ is not smaller than the Wiener constant
$$
W_K(A):=\inf\limits_{\mu \in {\mathfrak M}(A)} \ili_A U_K^\mu(y)\textup{d}\mu(y).
$$
In the case when $K$ is a positive, symmetric kernel satisfying a maximum principle, Farkas and Nagy \cite{Farkas2008} proved that $W_K(A)=T_K(A)$.  

While the assertions \eqref{s<d} and \eqref{mu^s} clearly indicate a connection between the discrete and continuous polarization problems,
what is yet to be fully understood is the limiting behavior (as $N \to \infty$) of the optimal $N$-point $K$-polarization configurations.
For continuous kernels, it is easy to establish (see \cite{Farkas2008},\cite{Farkas2006}, \cite{Farkas2006a}) that every weak-star limit of the normalized counting measures
associated with these $N$-point configurations must be an optimal (maximal) measure for the continuous polarization problem. However,
for other integrable kernels such as Riesz $s$-kernels when $s<\dim(A)$, only partial results are known (see  \cite{Simanek2015} and \cite{reznikov2016minimum}). For nonintegrable kernels, although the continuous problem is vacuous ($T_K(A)=\infty$), the asymptotic behavior of
optimal $N$-point discrete polarization configurations is a valid concern, especially in light of its connection
to the best-covering problem for large values of $s$ as mentioned above.\

Hereafter, our focus is on Riesz potentials so, for the sake of brevity, we write $\PP_s(A;N)$ in place of $\PP_{K_s}(A;N)$, and similarly for $P_s(A;\omega_N)$ and $\EE_s(A;N).$
The order of growth of the quantity $\mathcal P_s(A;N)$ in the case $s\geqslant \dim(A)$ was established by Erd{\'e}lyi and Saff \cite [Theorems 2.3 and 2.4]{Erdelyi2013}. If the $d$-dimensional Hausdorff measure of $A$ is positive, then
\begin {equation}\label {upper}
\mathcal P_s(A;N)=O(N^{s/d}),\ \ s>d, \ \ {\rm and}\ \ \mathcal P_d(A;N)=O(N\log N), \ \ N\to \infty.
\end {equation}
When $s=d$ and $A$ is a compact subset of a $d$-dimensional $C^1$-manifold, the following precise limit was established by Borodachov and Bosuwan \cite {Borodachov2014}:
\begin {equation} \label {s=d}
\lim\limits_{N\to \infty}\frac {\mathcal P_d(A;N)}{N\log N}=\frac {\textup{Vol}(\mathbb{B}^d)}{\mathcal H_d(A)},
\end {equation}
where $\mathbb{B}^d:=\{x\in \R^d\colon |x|\leqslant 1\}$ and by $\H_d$ we denote the $d$-dimensional Hausdorff measure on $\R^p$, scaled so that $\H_d({\mathcal Q})=1$, where $\mathcal{Q}$ is a $d$-dimensional unit cube embedded in $\R^p$. The cases $A= \mathbb{B}^d$ and $A=\S^d$ of \eqref {s=d} were earlier established in \cite {Erdelyi2013}.

Here we establish precise asymptotics for the case $s>d:=\dim(A)$. Specifically, as a consequence   of our main theorem, Theorem \ref{mainmain}, we show that for $s>d$, there exists a positive finite constant $\sigma_{s,d}$ such that  for a general class of $d$-dimensional  sets $A$ with $\H_d(A)>0$  we have the following limit:
$$
\lim_{N\to \infty}\frac{\mathcal{P}_s(A; N)}{N^{s/d}}=\frac{\sigma_{s,d}}{\H_d(A)^{s/d}}.
$$
Furthermore,    $N$-point $s$-polarization optimal  configurations are asymptotically uniformly distributed on $A$ with respect to $d$-dimensional Hausdorff measure.    We also consider in Theorem~\ref{mainmain} the   more general class of weighted Riesz potentials.

The paper is structured as follows. In Section \ref {S2} we present and discuss two important special cases, Theorem~\ref {mainmainmainmain}  and
Theorem~\ref{mainmainmainweight}, of our main result Theorem~\ref {mainmain}.  We illustrate these special cases with   the  examples of a smooth curve, a sphere, and a ball. Section~\ref {S2.5} contains  relevant definitions and the statement   of our main result. Section \ref {S3} compares our results with their known analogues for the minimal discrete Riesz energy, while the remaining sections are devoted to the proofs of our results.

\section{Some special cases of main result}\label{S2}

We begin with the following definition and some needed notation.
\begin{defin}\label{asympconf}
Assume $A\subset \R^p$ and $s>0$. For every positive integer $N$, let $\omega_N$ denote an $N$-point   configuration on $A$. We call a sequence $\{\omega_N\}_{N\geqslant 1}$   {\it asymptotically $s$-optimal} if
$$
\lim_{N\to \infty}\frac{P_s(A; \omega_N)}{\mathcal P_s(A;N)}=1.
$$
\end{defin}
Furthermore, by $\L_p$ we denote the Lebesgue measure on $\R^p$. If $x\in \R^p$ and $r>0$, by $B(x,r)$ we denote the open ball $\{y\in \R^p\colon |y-x|<r\}$ and by $B[x,r]$ the closed ball $\{y\in \R^p : \left|y-x\right|\leqslant r\}$.

Our first result concerns the asymptotic behavior of $\PP_s(A;N)$ as well as the associated optimal configurations.  In the statement we shall use the notion of {\em weak-star} convergence of discrete measures. For an $N$-point configuration $\omega_N$ on $A$ we associate the normalized counting measure 
\begin{equation}\nu(\omega_N):=\frac{1}{N}\sum_{x\in \omega_N}\delta_x,
\end{equation}
where $\delta_x$ denotes the unit point mass at $x$.  Recall that  $\nu(\omega_N)$ converges weak-star to a Borel probability measure $\mu$ on $A$ (and we write $\nu(\omega_N) \stackrel{*}{\longrightarrow} \mu$) if 
\begin{equation}
\lim_{N\to \infty}\int f \ {\rm d}\nu(\omega_N)= \lim_{N\to \infty} \frac{1}{N}\sum_{x\in \omega_N}f(x)=\int f \ {\rm d}\mu,
\end{equation}
for any $f\in C(A)$ or, equivalently (cf. \cite[Theorem 1.9.3]{Borodachov2015}), if
\begin{equation}
\nu(\omega_N)(B)=\#(\omega_N\cap B)/N\to \mu(B)\quad \text{ as } \quad N\to \infty,
\end{equation}
for any Borel measurable set $B\subset A$ with the $\mu(\partial B)=0$.

\begin{theorem}\label{mainmainmainmain}
Let $\mathcal{Q}_p$ denote the unit cube $[0,1]^p$ in $\R^p$. Then, for every $s>p$,  the limit
\begin{equation}\label{cube222}
\sigma_{s,p}:=\lim_{N\to \infty} \frac{\PP_s(\mathcal{Q}_p; N)}{N^{s/p}}
\end{equation}
exists and is positive and finite. More generally, if $s> d$ and $A$ is a compact subset of a   $d$-dimensional $C^1$ manifold in $\R^p$ with the relative boundary of $A$ having $\H_d$ measure  zero, then
$$
\lim_{N\to \infty} \frac{\PP_s(A; N)}{N^{s/d}} = \frac{\sigma_{s,d}}{\H_d(A)^{s/d}}.
$$
Furthermore, if $\H_d(A)>0$, then for any asymptotically $s$-optimal sequence $ \{\omega_N\}_{N\geqslant 1}$,
\begin{equation}\label{maindistr44444}
 \nu(\omega_N) \stackrel{*}{\longrightarrow}  \frac {1}{\H_d(A)}\H_d{\raisemath{-2pt}{\big |}}_A  \quad  \text{ as } \quad N\to \infty.
\end{equation}
\end{theorem}

We remark that in the special case of $d=p$, the theorem holds for {\em any} compact set $A\subset \R^p$ with $\L_p(\partial A)=0$.  
Establishing this special case plays a central role in the proof of our main theorem in Section~\ref{S2.5}.

Regarding the precise value of the  constant $\sigma_{s,p}$, for the case $p=1$ and $s>1$, Hardin, Kendall and Saff \cite{Hardin2013} proved that
$$
\sigma_{s,1}=2(2^s-1)\zeta(s), 
$$
where $\zeta(s)$ is the classical Riemann zeta-function. For $p=2$ we conjecture, based on the optimality properties of the equi-triangular lattice for the best-covering in $\R^2$, that the value of $\sigma_{s,2}$ for $s>2$ is
\begin{equation}\label{sigmaconj}
\sigma_{s,2}=\frac{3^{s/2}-1}{2}\zeta_\Lambda(s),
\end{equation}
where
$$
\zeta_\Lambda(s) := \sum_{v\in \Lambda\setminus\{0\}}\frac{1}{|v|^s},
$$
is the Epstein zeta-function for the equi-triangular lattice $\Lambda\subset \R^2$ with unit co-volume.   


\bigskip

We illustrate Theorem \ref{mainmainmainmain} with the following examples. 

\begin{example}
For a unit ball $\mathbb{B}^p\subset \R^p$ and $s> p$, Theorem~\ref{mainmainmainmain} asserts that
\begin{equation}\label{ball"}
\lim_{N\to \infty} \frac{\PP_s(\B^{p}; N)}{N^{s/p}}  = \sigma_{s,p}\cdot \left(\frac{\Gamma(p/2+1)}{\pi^{p/2}}\right)^{s/p}
\end{equation}
and, moreover,  for any asymptotically $s$-optimal sequence $\{\omega_N\}_{N\geqslant 1}$, 
\begin{equation}
\nu(\omega_N)\stackrel{*}{\longrightarrow}   \left(\frac{\Gamma(p/2+1)}{\pi^{p/2}}\right) \H_{p}{\raisemath{-2pt}{\big |}}_{\mathbb{B}^p}  \quad  \text{ as } \quad N\to \infty.
\end{equation}
 
 It is interesting to contrast the behavior in the hypersingular case with that for integrable Riesz kernels for the ball.  For $0<s\leqslant p-2$, Erd\'elyi and Saff  \cite{Erdelyi2013} show that for each $N$, the maximal $N$-point $s$-polarization configurations consist of $N$ points at the center of the ball (so
 $\PP_s(\B^{p}; N)=N$ for  $N\geqslant1$).
 For $p-2<s<p$, Simanek \cite{Simanek2015} has shown that the limiting distribution of optimal polarization configurations is the $s$-equilibrium measure for the corresponding minimal Riesz $s$-energy problem.  
\end{example}

\begin{example}
For a unit sphere $\S^{p-1}\subset \R^p$ and $s> p-1$, Theorem~\ref{mainmainmainmain} yields
\begin{equation}\label {sphere''}
\lim_{N\to \infty} \frac{\PP_s(\S^{p-1}; N)}{N^{s/(p-1)}}
 = \frac {\sigma_{s,p-1}}{\mathcal H_{p-1}(\S^{p-1})^{s/(p-1)}}=\sigma_{s,p-1}\cdot \left(\frac{\Gamma(p/2)}{2\pi^{p/2}}\right)^{s/(p-1)},
\end{equation}
and that, for   any asymptotically $s$-optimal sequence $\{\omega_N\}_{N\geqslant 1}$, 
\begin{equation}
\nu(\omega_N)\stackrel{*}{\longrightarrow} \left(\frac{\Gamma(p/2)}{2\pi^{p/2}}\right)\H_{p-1}{\raisemath{-2pt}{\big |}}_{\S^{p-1}}  \quad  \text{ as } \quad N\to \infty.
\end{equation}
  
  For the integrable Riesz kernel; that is, $0<s<p-1$,  it is shown in \cite{Simanek2015}   that the limiting distribution of optimal polarization configurations is the normalized surface area measure on the sphere.  Also, see \cite{reznikov2016minimum} for related results. 
\end{example}

\begin{example}
For any $C^1$-smooth curve $\Gamma$ with $0<\H_1(\Gamma)<\infty$ and any $s> 1$,  Theorem~\ref{mainmainmainmain} gives
\begin{equation}
\lim_{N\to \infty} \frac{\PP_s(\Gamma; N)}{N^{s}} =\frac{2(2^s-1)\zeta(s)}{\H_1(\Gamma)^{s}}.
\end{equation}
\end{example}

In \cite{Borodachov2014}, it is established that for the case $s=1$, that the limiting distribution of optimal $s$-polarization configurations
on smooth curves is normalized arclength measure, while
for the case of  integrable Riesz kernels on smooth curves,    every limit distribution of optimal polarization configurations is a solution to the continuous $s$-polarization problem \cite{reznikov2016minimum}.

We next turn to an extension of Theorem~\ref{mainmainmainmain} where we introduce a weight function.  
 For a function $w\colon A\times A\rightarrow [0,\infty]$, an $N$-point multiset $\omega_N=\{x_1, \ldots, x_N\}\subset A$ and $B\subset A$, we set
 \begin {equation}\label {wpot}
U^w_s(y; \omega_N):= \sum_{j=1}^N \frac{w(y,x_j)}{|y-x_j|^s},\qquad (y\in A),
\end {equation}
\begin {equation}\label {a}
P^w_s(B; \omega_N):=\inf_{y\in B}  U^w_s(y; \omega_N),
\end {equation}
and, define the {\it weighted $N$-th $(s,w)$-polarization (or Chebyshev) constant of $A$} by
\begin {equation}\label {Setting}
\PP^w_{s}(A; N):=\sup_{\omega_N\subset A} P^w_s(A; \omega_N).
\end {equation}

  In terms of the injector/dosage model discussed in Section~\ref{S1},  a weight function 
 can be used to introduce spatial inhomogeneity into  the strength of the sources as well as the dosage constraint. 
 For example, consider  $w(x,y)$ of the form $u(x)/v(y)$ for some positive, continuous functions  $u$ and $v$ on $A$.  Since  
   \begin{equation}
U^w_s(y; \omega_N)=\frac{1}{v(y)}U^{1\otimes u}_s(y; \omega_N),
\end{equation}
(where $1\otimes u(x,y)=u(x)$ for $x,y\in A$) the $N$-point $(s,w)$-polarization problem can be recast as locating  $N$ sources at points $x_k\in A$ of `strength' $u(x_k)$  so as to maximize the constant $C$ such that the `dosage' $U^{1\otimes u}_s(y; \omega_N)$  is at least $C v(y)$ for each $y\in A$.
Theorem~\ref{mainmainmainweight}  below states that the limiting density of sources as $N\to \infty$ for this weighted problem as the number sources goes is proportional to $(v(x)/u(x))^{d/s} \textup{d}\H_d(x)$.

We note that if $A$ is a compact set and the weight $w$ is lower semi-continuous and strictly positive on $A\times A$, then for any $N$ there exists a configuration $\omega_N^*=\{x_1^*, \ldots, x_N^*\}$ and a point $y^*$ such that
$$
\PP^w_s(A;N)=P^w_s(A; \omega_N^*) =U^w_s(y^*; \omega_N^*).
$$
For such a configuration, the potential
$
U^{w}_s(y): = U^w_s(y; \omega_N^*),
$
is called  an {\em optimal $N$-point   Riesz $(s,w)$-potential} for $A$.  Similarly to the unweighted case, we say that a sequence 
$\{\omega_N\}_{N\geqslant 1}$  of $N$-point configurations in $A$ is {\em asymptotically $(s,w)$-optimal } if
$$
\lim_{N\to \infty}\frac{P_s^w(A; \omega_N)}{\mathcal P_s^w(A;N)}=1.
$$

Our second consequence of Theorem~\ref{mainmain}   concerns the asymptotic behavior of $\PP^w_s(A;N)$ for a class of weights $w$. Denote
\begin{equation}
\tau_{s,d}(N):=\begin{cases}N^{s/d}, \; &s>d, \\ N\log N, \; &s=d. \end{cases}
\end{equation}
We prove the following.
\begin{theorem}\label{mainmainmainweight}
Let $d$ and $p$ be positive integers with $d\leqslant p$. Suppose $A\subset \R^p$ is a compact subset of a $d$-dimensional $C^1$-manifold with $\H_d(\partial A)=0$, and $w\in C(A\times A)$ with $w(x,x)$  positive for all   $x\in A$. Then for any $s\geqslant d$, 
\begin{equation}\label {h_sp111}
\lim_{N\to\infty}\frac{\PP^w_s(A; N)}{\tau_{s,d}(N)}=\frac{\sigma_{s,d}}{[\H^{s,w}_d(A)]^{s/d}},
\end{equation}
where, for any measurable $B\subset\R^p$,  
\begin{equation}\label{Hswdef}
\H^{s,w}_d(B):=\int_{B\cap A}w^{-d/s}(x,x) \textup{d}\H_d(x)
\end{equation}
and $\sigma_{s,d}$ for $s>d$ is as in Theorem~\ref{mainmainmainmain} and $\sigma_{d,d}:=\text{{\rm Vol}}({\mathbb{B}}^d)$.  
Moreover, if $\H_d(A)>0$, then for any asymptotically $(s,w)$-optimal sequence $\{\omega_N\}_{N\geqslant 1}$, 
\begin{equation}\label{maindistr111}
\nu(\omega_N)\stackrel{*}{\longrightarrow} \frac{1}{\H^{s,w}_d(A)} \H^{s,w}_d  \quad  \text{ as } \quad N\to \infty.
\end{equation}
\end{theorem}

\bigskip


\section {Statement of main result}
\label{S2.5}

In this section we state our main theorem.   For this purpose we first introduce some needed definitions and notation concerning geometric properties of the set
$A$ as well as continuity and positivity properties of the considered weight   $w$.


\begin{defin}\label{bilip}
A function $\phi:A\subset \R^p\to \R^d$ is said to be {\em bi-Lipschitz with constant $C$} if
$$
C^{-1}|x-y|\leqslant |\phi(x)-\phi(y)|\leqslant C|x-y|, \qquad (x,y)\in A,
$$
while $\phi$ is said to be  {\em Lipschitz with constant $C$} if the second inequality above holds.

A set $A\subset \R^p$ is called {\it $(\H_d, d)$-rectifiable}, $d\leqslant p$, if $\H_d(A)<\infty$ and $A$ is the union of at most countably many images of bounded sets in $\R^d$ under Lipschitz maps  and a set of $\H_d$-measure zero (see \cite{Mattila1995}).

Further, we say that $A$ is {\it $d$-bi-Lipschitz at $x\in A$} if, for any $\ep>0$, there exists a number $\delta>0$,  and a bi-Lipschitz function $\vf_{x, \epsilon}\colon B(x, \delta)\cap A\rightarrow \R^d$ with constant $(1+\epsilon)$ such that the set 
$ \vf_{x, \epsilon}(B(x,\delta)\cap A)\subset \R^d$ is open.

By $A_{\rm bi}$ we denote the set of all points $x\in A$ at which $A$ is $d$-bi-Lipschitz. Further, denote $A_{\rm bi}^c:=A\setminus A_{\rm bi}$.
\end{defin}

Notice that any set $A\subset \R^p$ is $(\H_p, p)$-rectifiable with $A_{\rm bi}^c=\partial A$. We remark that any compact set $A$ with $\H_d(A)<\infty$ and $\H_d(A_{\rm bi}^c)=0$ is $(\H_d, d)$-rectifiable. Thus, any embedded compact $C^1$-smooth $d$-dimensional manifold with $\H_d(\partial A)=0$ is $(\H_d, d)$-rectifiable. In particular, if this manifold is closed, then $A_{\rm bi}^c=\emptyset$. Further, a finite union of $C^1$-smooth arcs is an $(\H_1, 1)$-rectifiable set.

The following notion of Minkowski content often arises in   geometric measure theory.
\begin{defin}\label{content}
Let $A\subset \R^p$ be a bounded set, $A(\epsilon):=\{x\in \R^p\colon \dist(x, A)< \ep\}$ and, for $m\geqslant 1$, let $\beta_{m}$ denote the volume of the $m$-dimensional unit ball (we also set $\beta_0:=1$). If the limit
$$
\mathcal{M}_d(A):=\lim_{\ep\to 0^+}\frac{\L_p(A(\epsilon))}{\beta_{p-d}\ep^{p-d}}
$$
exists, then it is called the {\it $d$-Minkowski content} of $A$.
\end{defin}

We remark that the notion of Minkowski content has been particularly useful in the study of discrete $s$-energy where   the equality $\H_d(A)=\mathcal{M}_d(A)$ plays an important role in the proof of asymptotic results; see Theorem~\ref{thenergy}.

We equip the set $A\times A$ with the metric
$$
{\rm dist}((x_1,y_1),(x_2,y_2))=\sqrt {\left|x_1-x_2\right|^2+\left|y_1-y_2\right|^2},
$$ where $x_1,x_2,y_1,y_2\in A$.
Concerning the weight $w(x,y)$ we utilize the following definition from \cite{Borodachov2008}.

\begin{defin}\label{CPDweight}
Suppose $A\subset \R^p$ is a compact set. We call a function $w\colon A\times A\to [0, \infty]$ a {\it CPD-weight \footnote{Here CPD stands for (almost) continuous and positive on the diagonal} on $A\times A$ with parameter $d$} if the following properties hold:
\begin{enumerate}
\item $w$ is continuous (as a function on $A\times A$) at $\H_d$-almost every point of the diagonal $D(A):=\{(x,x)\colon x\in A\}$;
\item there is a neighborhood $G$ of $D(A)$ (relative to $A\times A$), such that $\inf_G w(x,y)>0$;
\item $w$ is bounded on any closed subset $B\subset A\times A$ with $B\cap D(A)=\emptyset$.
\end{enumerate}
\end{defin}
In what follows, we define
\begin{equation}
\u{h}^w_{s,d}(A):=\liminf_{N\to\infty} \frac{\PP^w_{s}(A; N)}{\tau_{s,d}(N)}, \; \; \o{h}^w_{s,d}(A):=\limsup_{N\to \infty} \frac{\PP^w_{s}(A; N)}{\tau_{s,d}(N)}.
\end{equation}
If $\u{h}^w_{s,d}(A)=\o{h}^w_{s,d}(A)$, we denote
\begin{equation}
h^w_{s,d}(A):=\lim_{N\to \infty} \frac{\PP^w_s(A; N)}{\tau_{s,d}(N)}.
\end{equation}
If the function $w$ is identically equal to $1$, we drop the superscript and write $\u{h}_{s,d}$, $\o{h}_{s,d}$, and $h_{s,d}$.

We are ready to state our most general theorem.
\begin{theorem}\label{mainmain}
Let $d$ and $p$ be positive integers with $d\leqslant p$. Suppose $A\subset \R^p$ is a compact set with $\H_d(A)=\mathcal{M}_d(A)<\infty$ and $\H_d(\clos(A_{\rm bi}^c))=0$. Assume $w$ is a CPD-weight on $A\times A$ with parameter $d$. Then for any $s\geqslant d$, 
\begin{equation}\label {h_sp}
h^w_{s,d}(A)=\lim_{N\to\infty}\frac{\PP^w_s(A; N)}{\tau_{s,d}(N)}=\frac{\sigma_{s,d}}{[\H^{s,w}_d(A)]^{s/d}}.
\end{equation}
Moreover, if $\H_d(A)>0$, then for any asymptotically $(s,w)$-optimal sequence $\{\omega_N\}_{N\geqslant 1}$, 
\begin{equation}\label{maindistr}
\nu(\omega_N)\stackrel{*}{\longrightarrow} \frac{1}{\H^{s,w}_d(A)} \H^{s,w}_d  \quad  \text{ as } \quad N\to \infty.
\end{equation}
\end{theorem}
In the case $w=1$ and $s=d$ (recall that $\sigma_{d,d}=\text{Vol}(\mathbb B^d)=\beta_d$), Borodachov and Bosuwan \cite{Borodachov2014} proved the above theorem for sets $A=\cup_{j=1}^{m} A_j$, where each $A_j$ is a compact subset of a $C^1$-smooth $d$-dimensional manifold in $\R^p$, with $\mathcal H_d(A_j\cap A_k)=0$ if $j\not=k$. 

We remark that the equality $\H_d(A)=\mathcal{M}_d(A)$ holds if $A$ is a {\it $d$-rectifiable} compact set; that is, $A$ is the image of a compact subset of $\R^d$ under a Lipschitz map (in particular, this equality holds if $d=p$). Moreover, if $A$ is $(\H_d, d)$-rectifiable with $\H_d(A)=\mathcal{M}_d(A)$, then the same is true for every compact subset of $A$. For details, see \cite[Chapter 7]{Borodachov2015}.

We further remark that any embedded $d$-dimensional compact $C^1$-smooth manifold $A$ with $\H_d(\partial A)=0$ satisfies conditions of the theorem. Moreover, any finite union of $C^1$-smooth arcs also satisfies these conditions. On the other hand, a ``fat'' Cantor set $\mathcal{C}\subset [0,1]$ with $\H_1(\mathcal{C})>0$ (thus, of dimension 1) does not satisfy the condition $\H_1(\mathcal{C}_{\rm bi}^c)=0$.

\section{Comparison with energy asymptotics}\label {S3}

In this section we provide a sufficient condition for   $h_{s,d}^w(A)$ to be infinite when     $s>d$ and   sets $A$ that are sufficiently small (see Corollary~\ref{measurezero}).  First we recall  a result concerning the asymptotics of weighted discrete energy in the hyper-singular case $s\geqslant d$. For a compact set $A\subset \R^p$, weight $w\colon A\times A\to [0,\infty]$ and an integer $N\geqslant 2$, define
$$
\EE^w_s(A; N):=\inf\left\{ \sli_{\substack{x, y\in \omega_N \\ x\not= y}} \frac{w(x, y)}{|x-y|^s} \colon \omega_N\subset A, \, \#\omega_N=N\right\}.
$$
If the weight $w$ is identically equal to $1$, we drop the superscript $w$. For an infinite set $A$, any $s>0$, and a non-negative weight $w$ on $A\times A$ we, similar to \cite[Theorem 2.3]{Erdelyi2013}), obtain
\begin{equation}\label{polarenergy}
\PP^w_s(A; N)\geqslant \frac{\EE^w_s(A;N)}{N-1}, \; \; \; \; \;  N\geqslant 2	.
\end{equation}

The following theorem, proved by Borodachov, Hardin and Saff, \cite{Borodachov2008, Borodachov2015}, describes the asymptotic behavior of $\EE^w_s(A;N)$.
\begin{theorem}\label{thenergy}
Let $d$ and $p$ be positive integers with $d\leqslant p$. Suppose $A\subset \R^p$ is a compact $(\H_d, d)$-rectifiable set with $\mathcal{M}_d(A)=\H_d(A)$ and $w$ is a CPD-weight on $A\times A$ with parameter $d$. If $s>d$, then
for any compact set $B\subset A$,
$$
\lim_{N\to \infty}\frac{\EE^w_s(B; N)}{N^{1+s/d}} = \frac{C_{s,d}}{[\H_d^{s,w}(B)]^{s/d}},
$$
where $C_{s,d}$ is a finite positive constant that depends only on $s$ and $d$. If $A$ is a compact subset of a $d$-dimensional $C^1$-smooth manifold, then for any compact set $B\subset A$,
$$
\lim_{N\to \infty}\frac{\EE^w_d(B; N)}{N^2\log N} = \frac{\beta_d}{\H_d^{d,w}(B)},
$$
where $\beta_d=\text{Vol}(\mathbb{B}^d)$.

In particular, if $d=p$ and $A\subset \R^p$ is a compact set with $\L_p(A)=0$, then both limits above are equal to $\infty$.
\end{theorem}

The following corollary of Theorem \ref {thenergy} proves a particular case of Theorem \ref{mainmain} and will be used in the proof of Theorem~\ref{lemmafromabove}.
\begin{corollary}\label{measurezero}
If $A\subset \R^p$ is a compact set with $\H_d(A)=\mathcal{M}_d(A)=0$ and $w$ is a CPD-weight on $A$ with parameter $d$, then
$$
h_{s,d}^w(A)=\lim_{N\to\infty}\frac{\PP^w_s(A; N)}{\tau_{s,d}(N)}=\infty.
$$
\end{corollary}
\begin{proof}
Dividing both sides of  \eqref{polarenergy} by $\tau_{s,d}(N)$ and using Theorem~\ref {thenergy}, we obtain
$$
\u{h}^w_{s,d}(A) \geqslant \lim_{N\to \infty}\frac{\EE^w_s(A;N)}{(N-1)\tau_{s,d}(N)} = \infty. 
$$
\end{proof}

\section{Proofs}

 The remaining sections are devoted to the proof of our main result Theorem~\ref{mainmain}.  In Section~\ref{S5} we determine  the  dominant asymptotic term of  $\PP_s(A; N)$ as $N\to \infty$ for the unit cube $A=\mathcal{Q}_p$; that is, we establish that equation \eqref{cube222} holds.     In Section~\ref{S6} we prove a subadditive property of $\u{h}^w_{s,d}(\cdot)$.  In  Section~\ref{S7} we use the subadditive property together   with \eqref{cube222} to first find a   lower bound for $\u{h}^w_{s,d}(A)$  for the case  that $A$ is a compact set in $\R^p$ of positive Lebesgue measure (see Lemma~\ref{lemmafromabove}) and then to generalize this lower bound to the case that $A$  is a sufficiently regular $d$-rectifiable set (see Lemma~\ref{weightedfrombelow}) embedded in $\R^p$.  In Section~\ref{S8},  we determine the limiting distribution of an asymptotically $(s,w)$-optimal sequence of $N$-point configurations and in the final section we establish an upper bound that proves that  the limit  ${h}^w_{s,d}(A)$  exists thereby completing the proof of Theorem~\ref{mainmain}.
  
In the rest of this section we collect some preliminary results that will be useful in the following proofs.  
First, we consider some basic properties  of $P_s^w(B;\omega_N)$ in terms of its arguments $B$ and $\omega_N$. These properties are  immediate consequences of the definition of $P_s^w$ given in \eqref{a}.  
\begin{lemma} \label{basiclemma}
Let $A$ be a compact set in $\R^p$, $w$ a function on $A\times A$ taking values in $[0,\infty]$, $B$ and $\tilde B$ subsets of $A$, and   $\omega_N$ and $\tilde\omega_M$ finite configurations in $A$. 
 
\begin{enumerate}
\item If $\tilde\omega_M\subset \omega_N$ and $\tilde B\supset B$, then
$$
P_s^w(B;\omega_N)\geqslant P_s^w(\tilde B,\tilde \omega_M).
$$
\item If $B_1,\ldots, B_k$ are subsets of $B$ such that $B=\bigcup_i B_i$, then
$$
P_s^w(B;\omega_N)=  \min_i P_s^w( B_i, \omega_N).
$$
\end{enumerate}
   \end{lemma}
   
In several of our later proofs we shall need the   existence of a sufficiently regular `Vitali-type' covering for subsets of $A_{\rm bi}$.  

 \begin{lemma}\label{vitali}
Let $A\subset \R^p$ be a compact set with $\H_d(A)<\infty$, and $B\subset A\setminus \clos(A_{\rm bi}^c)$  a nonempty set open relative to $A$.  For $\epsilon>0$, there exists a pairwise disjoint  collection $\mathcal{X}_{\ep}=\{Q_{\alpha}\}$ of closed sets $Q_\alpha:=B[x_\alpha,\rho_\alpha]\cap A$  such that 
\begin{equation}
\label{BQalpha}\displaystyle \H_d\left(B\setminus \mathsmaller{\bigcup}_{Q_\alpha\in \mathcal{X}_{\ep}} Q_{\alpha}\right)=0, 
\end{equation} and 
such that  for each $\alpha$, we have  $\rho_\alpha<  \epsilon$ and that there is some bi-Lipschitz   $\vf_\alpha$ with   constant $(1+\epsilon)$ mapping $Q_\alpha$ onto 
$\tilde Q_\alpha:=\vf_\alpha(Q_\alpha)$ such that $\mathcal{L}_d(\partial {\tilde Q}_\alpha)=0$ and 
\begin{equation}
\label{Qalphatilde}
\tilde Q_\alpha \supset B[\vf_\alpha(x_\alpha),\rho_\alpha/(1+\epsilon)].
\end{equation}
If $\epsilon>0$ and $\gamma>0$ then there is some finite collection 
$\mathcal{X}_{\epsilon,\gamma}\subset \mathcal{X_\epsilon}$ such that
\begin{equation}
\label{BQalpha2}\displaystyle \H_d\left(B\setminus \mathsmaller{\bigcup}_{Q_\alpha\in \mathcal{X}_{\ep,\gamma}} Q_{\alpha}\right)<\gamma.
\end{equation}
\end{lemma}
\begin{proof}
 
Let  $\epsilon>0$.  Since $B\subset A_{\rm bi}$ and $B$ is relatively open,    then for each $x\in B$ Definition~\ref{bilip}  implies that there is a number $\delta=\delta(x, \ep)>0$ and a bi-Lipschitz function $\vf_{x, \epsilon}\colon B(x,\delta)\cap B\to \R^d$ with   constant   $1+\ep$, such that $U_x:=\vf_{x, \epsilon}(B(x,\delta)\cap B)$ is  an  open set in $\R^d$.
Thus,  there exists some $r=r(x)>0$  so that $B(\vf_{x,\epsilon}(x), r)\subset U_x$ and, hence, using the fact that $\vf_{x, \epsilon}$ has bi-Lipschitz constant $(1+\epsilon)$, we have for $0<\rho<r(x)/(1+\epsilon)$ that 
 $Q_{x,\rho}:=B[x,\rho]\cap B\subset \vf^{-1}_{x, \epsilon}(B(\vf_{x,\epsilon}(x),  r))$ and so $\vf_{x, \epsilon}(Q_{x,\rho})\supset B[\vf_{x,\epsilon}(x),  \rho/(1+\epsilon)]$. 
 Let $$
V_{\ep}(B):=\left \{Q_{x, \rho}\colon 0<\rho\leqslant \min\left\{ {r(x)}/{(1+\ep)}, \ep\right\}, \; \; x\in B\right\}.
$$

 Then by Vitali's covering theorem for Radon measures (see, for example, \cite[Theorem 2.8]{Mattila1995}), there is a pairwise disjoint collection $\{Q_\alpha\}\subset V_{\ep}(B)$ such that
 \eqref{BQalpha} holds.     By construction each $Q_\alpha$ is of the form $B[x_\alpha,\rho_\alpha]\cap B$ and $\vf_\alpha:=\vf_{x_\alpha,\epsilon}{\raisemath{-2pt}{\big |}}_{Q_\alpha}$ is bi-Lipschitz with constant $(1+\epsilon)$ and such that \eqref{Qalphatilde} holds. 
 
For $\gamma>0$, the existence of such a finite collection $\mathcal{X}_{\epsilon,\gamma}$ satisfying \eqref{BQalpha2} follows from the fact that the elements of $X_\epsilon$ are pairwise disjoint and  that $\H_d(B)<\infty$. 
\end{proof}

\section{Proof of equality \eqref{cube222}}\label {S5}

In this section we prove that the limit $h_{s,p}(\mathcal{Q}_p)$ exists for any $s>p$ and that $\sigma_{s,p}=h_{s,p}(\mathcal{Q}_p)$ is a positive finite number. For the case $s=p$, this fact was proved by Borodachov and Bosuwan \cite{Borodachov2014} using a different method. Our proof for $s>p$ utilizes an argument similar to the one   in \cite{Hardin2005}.

\bigskip

For $N\in \NN$, let  $\omega_N  $ be an $s$-polarization optimal $N$-point configuration for $ \mathcal{Q}_p$; that is,
$  P_s(\mathcal{Q}_p; \omega_N)=\PP_s(\mathcal{Q}_p; N).$
For    $m\geqslant2$, $m\in \NN$,  and a vector ${\bf j} = (j_1, j_2, \ldots, j_p)\in \mathbb{Z}^p$ with $0\leqslant j_k \leqslant m-1$, define
$$
Q^{{\bf j}}:=\left[\frac{j_1}m, \frac{j_1+1}m\right] \times \cdots \times \left[\frac{j_p}m, \frac{j_p+1}m\right] = \frac{1}m(\mathcal{Q}_p + {\bf j}),
$$
$$
\omega_N^{{\bf j}}:= \frac{1}m (\omega_N+ {\bf j})\subset Q^{{\bf j}},
$$
and $\overline\omega_{m^p N}:=\bigcup_{{\bf j}} \omega_N^{{\bf j}}\subset \mathcal{Q}_p$.
Then, using Lemma~\ref{basiclemma}, we obtain
\begin{multline}\label{QmpN1}
\PP_s(\mathcal{Q}_p; m^p N) \geqslant P_s(\mathcal{Q}_p; \overline\omega_{m^p N}) = \min_{{\bf j}} P_s(Q^{{\bf j}}; \overline\omega_{m^p N})   \geqslant  \min_{{\bf j}} P_s(Q^{{\bf j}};\omega_{N}^{{\bf j}})=P_s(Q^{{\bf 0}}; \omega_N^{{\bf 0}}),
\end{multline}
where the last equality follows from the observation that 
$P_s(Q^{{\bf j}}; \omega_N^{{\bf j}})=P_s(Q^{{\bf 0}}; \omega_N^{\bf 0})$   since $Q^{{\bf j}}$ and  $\omega_N^{{\bf j}}$ are translations by ${\bf j}/m$ of $Q^{{\bf 0}}$ and  $\omega_N^{{\bf 0}}$, respectively.  Furthermore,  the scaling relations $ Q^{{\bf 0}}=(1/m)\mathcal{Q}_p$ and $  \omega_N^{\bf 0}=(1/m)\omega_N$   together with  \eqref{QmpN1} imply 
\begin{equation}\label{QmpN1a}
\PP_s(\mathcal{Q}_p; m^p N) \geqslant     P_s(Q^{{\bf 0}}; \omega_N^{{\bf 0}})=m^s P_s(\mathcal{Q}_p; \omega_N) \geqslant m^s \PP_s(\mathcal{Q}_p; N) .
\end{equation}

From inequality \eqref {upper} we have $\overline h_{s,p}(\mathcal {Q}_p)<\infty$.
Let $\epsilon>0$ and let  $N_0$ be a positive integer such that
$$
\frac{\PP_s(\mathcal{Q}_p; N_0)}{N_0^{s/p}} > \o{h}_{s,p}(\mathcal{Q}_p)-\ep.
$$
For   $N>N_0$ choose the non-negative integer $m_N$ such that $m_N^p N_0 \leqslant N < (m_N +1)^p N_0$.
Then, from \eqref{QmpN1a} we get
$$
\o{h}_{s,p}(\mathcal{Q}_p) < \frac{\PP_s(\mathcal{Q}_p; N_0)}{N_0^{s/p}} + \ep = \frac{m_N^s\PP_s(\mathcal{Q}_p; N_0)}{m_N^s N_0^{s/p}} + \ep \leqslant \frac{\PP_s(\mathcal{Q}_p; m_N^p N_0)}{m_N^s N_0^{s/p}} + \ep.
$$
Notice that the inequality $m_N^p N_0 \leqslant N$ implies $\PP_s(\mathcal{Q}_p; m_N^p N_0) \leqslant \PP_s(\mathcal{Q}_p; N)$. Therefore,
\begin{equation} \label{QmpN2}
\o{h}_{s,p}(\mathcal{Q}_p) < \frac{\PP_s(\mathcal{Q}_p; N)}{N^{s/p}} \cdot \left(\frac{m_N+1}{m_N}\right)^s + \ep.
\end{equation}
Taking the limit inferior as $N\to\infty$ in \eqref{QmpN2} and noting that $m_N\to \infty$ as $N\to \infty$, we obtain
\begin{equation}\label{PsQplower}
\o{h}_{s,p}(\mathcal{Q}_p) \leqslant \u{h}_{s,p}(\mathcal{Q}_p)+\ep.
\end{equation}
In view of the arbitrariness of $\ep$, the limit $\sigma_{s,p}:=h_{s,p}(\mathcal{Q}_p)$ exists as a finite real number. Inequality \eqref{polarenergy} together with Theorem \ref{thenergy}  imply that $ \sigma_{s,p}$ is positive.
\hfill $\Box$

\medskip

One may alternatively prove the positivity of $\sigma_{s,p}$ directly without using Theorem~\ref {thenergy}.  One method consists of   dividing the cube $\mathcal Q_p$ into $N=n^p$ equal subcubes and letting $\omega_N$ be the configuration consisting of the centers of these cubes, then it is not difficult to prove that $P_s(\mathcal Q_p;\omega_N)$ will have order $N^{s/d}$ as $N\to\infty$.

\section{Sub-additivity of $[\u{h}_{s,d}^w(\cdot)]^{-d/s}$}\label {S6}

The following lemma establishes the sub-additivity of $[\u{h}_{s,d}^w(\cdot)]^{-d/s}$ and  will play an important role in the proof of   \eqref {h_sp}, see Lemmas \ref{lemmafromabove} and \ref{weightedfrombelow}.
\begin{lemma}\label{addit}
Suppose $B$ and $C$ are subsets of $\R^p$ and $w\colon (B\cup C)\times (B\cup C) \to [0,\infty]$. Then for any positive $d\leqslant p$ and any $s\geqslant d$, 
\begin{equation}\label {addit'}
\u{h}^w_{s,d}(B\cup C)^{-d/s}\leqslant \u{h}^w_{s,d}(B)^{-d/s}+\u{h}^w_{s,d}(C)^{-d/s}.
\end{equation}
\end{lemma}
\begin{proof}
First note that for any $N$-point configuration $\omega_N\subset B\cup C$, Lemma~\ref{basiclemma} gives
\begin{equation}
P_s^w(B\cup C; \omega_N)=\min\left\{P_s^w(B,\omega_N),P_s^w(C,\omega_N)\right\}\geqslant \min\left\{ P_s^w(B; \omega_N\cap B),P_s^w(  C; \omega_N\cap C)\right\}.
\end{equation}
If  $N_1$, $N_2$, and $N$ are positive integers such that $N_1+N_2=N$  then,  with $\omega_N$ denoting an arbitrary $N$ point configuration in $B\cup C$, 
we have
\begin{equation}\label{koaks}
\begin{split}
\PP^w_s(B \cup C; N)&=\sup_{\omega_N}\left( P_s^w(B\cup C; \omega_N) \right)\\ &\geqslant
\sup_{\substack{
\#\omega_N\cap B\geqslant N_1\\  
\#\omega_N\cap C\geqslant N_2}}\min\left( P_s^w(B; \omega_N\cap B),P_s^w(  C; \omega_N\cap C)\right)\\&
\geqslant \min \left\{\PP^w_s(B; N_1), \PP^w_s(C; N_2)\right\},
\end{split}
\end{equation}

 We now assign particular values to $N_1$ and $N_2$. For  a fixed   $\alpha\in (0,1)$ and $N\in \NN$, let  $N_1:=\lfloor \alpha N\rfloor$ and $N_2:=N-N_1$  and note that $N_1\to \infty$ and $N_2\to \infty$ as $N\to\infty$. Then the inequality  \eqref{koaks} implies
$$
\frac{\PP^w_s(B\cup C; N)}{\tau_{s,d}(N)} \geqslant \min \left\{\frac{\PP^w_s(B; N_1)}{\tau_{s,d}(N_1)}\cdot \frac{\tau_{s,d}(N_1)}{\tau_{s,d}(N)}, \frac{\PP^w_s(C; N_2)}{\tau_{s,d}(N_2)}\cdot \frac{\tau_{s,d}(N_2)}{\tau_{s,d}(N)} \right\},
$$
which, together with
$$
\lim_{N\to \infty} \frac{\tau_{s,d}(N_1)}{\tau_{s,d}(N)} = \alpha^{s/d}, \;\;\;\; \lim_{N\to \infty} \frac{\tau_{s,d}(N_2)}{\tau_{s,d}(N)}=(1-\alpha)^{s/d},\ \ \ s\geqslant d,
$$
yields
\begin{equation}\label{kvaks}
\u{h}_{s,d}^w(B\cup C) \geqslant \min \left\{\alpha^{s/d}\u{h}_{s,d}^w(B), (1-\alpha)^{s/d} \u{h}_{s,d}^w(C) \right\} \; \; \; \mbox{for any $\alpha\in(0,1)$}.
\end{equation}
If $\u{h}_{s,d}^w(B)=0$ or $\u{h}_{s,d}^w(C)=0$ then \eqref{addit'} holds trivially and so  we assume  both $\u{h}_{s,d}^w(B)$ and $\u{h}_{s,d}^w(C)$ are positive.
If $\u{h}_{s,d}^w(B)=\u{h}_{s,d}^w(C)=\infty$ then the right-hand side of \eqref {kvaks} is equal to $\infty$, and the lemma holds trivially. If $\u{h}_{s,d}^w(B)<\infty$ and $\u{h}_{s,d}^w(C)=\infty$ then the right-hand side of \eqref {kvaks} is equal to $\alpha^{s/d}\u{h}_{s,d}^w(B)$. Letting $\alpha$ go to $1$ we obtain the lemma. The case $\underline h_{s,d}^w(B)=\infty$ and $\underline h_{s,d}^w(C)<\infty$ is treated similarly.

If both $\u{h}_{s,d}^w(B)$ and $\u{h}_{s,d}^w(C)$ are positive and finite, then we set
$$
\alpha:=\frac{\u{h}_{s,d}^w(C)^{d/s}}{\u{h}_{s,d}^w(B)^{d/s}+\u{h}_{s,d}^w(C)^{d/s}} \in (0,1).
$$
This choice of $\alpha$ together with inequality \eqref{kvaks} implies the estimate \eqref{addit'}.
\end{proof}

\section{An estimate of $\u{h}^w_{s,d}(A)$ from below}\label {S7}

In this section we prove important corollaries of Lemma \ref{addit}. We start with the unweighted case (i.e., $w=1$) for $d=p$.
\begin{lemma}\label{lemmafromabove}
Suppose $A\subset \R^p$ is a compact set with $\L_p(\partial A)=0$. Then for any $s\geqslant p$,
\begin{equation}\label{frombelownoweight}
\u{h}_{s,p}(A)\geqslant \frac{\sigma_{s,p}}{\L_p(A)^{s/p}}.
\end{equation}
\end{lemma}
\begin{proof}
If $\L_p(A)=0$ then the lemma follows from Corollary \ref{measurezero}. Thus, we assume $\L_p(A)>0$.

Let $\ep>0$. Our assumptions on the set $A$ imply that there exists a finite family $\mathcal{D}=\{Q_i\}$ of closed cubes with disjoint interiors, such that $Q_i\subset A$ and
$$
\L_p\left(A\setminus  \mathsmaller{\bigcup}_{i}Q_i\right) < \ep.
$$
Denote $D:=A\setminus \cup_{i}Q_i$. Since $\L_p(\partial A)=0$, we also get $\L_p(\partial D)=0$. Thus, $\L_p(\textup{clos}(D))=\L_p(D)<\ep$. From inequality \eqref{polarenergy} and Theorem \ref{thenergy} we obtain
$$
\u{h}_{s,p}(\textup{clos}(D))\geqslant \lim_{N\to \infty}\frac{\EE_s(\textup{clos}(D);N)}{(N-1)\tau_{s,p}(N)} \geqslant C_{s,p}\ep^{-s/p}.
$$
Further, inequality \eqref{addit'} yields
$$
\u{h}_{s,p}(A)^{-p/s} \leqslant \sli_i \u{h}_{s,p}(Q_i)^{-p/s} + \u{h}_{s,p}(\textup{clos} (D))^{-p/s} \leqslant  \sli_i\u{h}_{s,p}(Q_i)^{-p/s} + C_{s,p}^{-p/s}\ep.
$$

Equality \eqref{cube222} implies that $\u{h}_{s,p}(Q_i) = \sigma_{s,p}\L_p(Q_i)^{-s/p}$. Thus,
\begin {equation}\label{ineq222}
\begin {split}
\u{h}_{s,p}(A)^{-p/s} &\leqslant \sli_i \sigma_{s,p}^{-p/s} \L_p(Q_i) + C_{s,p}^{-p/s}\ep \\
&= \sigma_{s,p}^{-p/s} \L_p\left( \mathsmaller{\bigcup}_i Q_i\right) + C_{s,p}^{-p/s}\ep \leqslant \sigma_{s,p}^{-p/s} \L_p(A) + C_{s,p}^{-p/s}\ep.
\end {split}
\end {equation}
Taking $\ep\to 0$ in \eqref{ineq222} then gives \eqref{frombelownoweight}.
\end{proof}
Next, we deduce a general estimate for $\u{h}^w_{s,d}$. Namely, we prove the following lemma.
\begin{lemma}\label{weightedfrombelow}
Suppose $d,p\in \NN$, $d\leqslant p$, $A\subset \R^p$ is a compact set with $\H_d(A)=\mathcal{M}_d(A)<\infty$ and $\H_d(\clos(A_{\rm bi}^c))=0$. Suppose $w$ is a CPD weight on $A\times A$ with parameter $d$. Then for any $s\geqslant d$, 
\begin{equation}\label{frombeloww}
\u{h}_{s,d}^w(A)\geqslant \frac{\sigma_{s,d}}{\H^{s,w}_d(A)^{s/d}}.
\end{equation}
\end{lemma}
\begin{proof}

Let $B:=A\setminus \clos(A_{\rm bi}^c)$ and note that  $B$ is a  subset of $A_{\rm bi}$ open relative to $A$. By assumption, $\clos(A_{\rm bi}^c)$ is a compact subset of $A$ of zero $\H_d$-measure. Then taking into account inequality \eqref{polarenergy} and Theorem \ref{thenergy} we obtain
\begin{equation}\label{reallysmallone}
\u{h}^w_{s,d}(\clos(A_{\rm bi}^c)) \geqslant \lim_{N\to \infty} \frac{\EE^w_{s}(\clos (A_{\rm bi}^c);N)}{(N-1)\tau_{s,d}(N)} = C_{s,d}[\H_d^{s,w}(\clos(A_{\rm bi}^c))]^{-s/d} = \infty.
\end{equation}
Let $\ep>0$ and let      $\mathcal{X}_{\ep, \ep}$ be a finite family of disjoint sets $\{Q_\alpha\}$ as in  Lemma \ref{vitali} with $\gamma=\epsilon$.  Define  
 $D:=B\setminus \cup_{\alpha} Q_\alpha$. Since $\clos(D)$ is a compact subset of $A$, inequality \eqref{polarenergy} and Theorem \ref{thenergy} imply
\begin{equation}\label{reallysmalltwo}
\u{h}^w_{s,d}(\clos(D))\geqslant \lim_{N\to \infty} \frac{\EE^w_{s}(\clos(D);N)}{(N-1)\tau_{s,d}(N)} = C_{s,d}[\H_d^{s,w}(\clos(D))]^{-s/d}.
\end{equation}

Next, we will estimate $\u{h}_{s,d}^w(Q_\alpha)$ for each $\alpha$. Recall 
$
\tilde{Q}_\alpha:=\vf_\alpha(Q_\alpha)
$ and that $\mathcal{L}_d(\partial \tilde{Q}_\alpha)=0$.
Let  $\tilde\omega_N $ denote an arbitrary $N$-point configuration in   $\tilde{Q}_\alpha$  and      $\omega_N:=\vf_\alpha^{-1}(\tilde\omega_N)\subset Q_\alpha$ denote the preimage of $\tilde \omega_N$. Set
$$
\u{w}_{Q_\alpha}:=\inf_{(x,y)\in Q_\alpha\times Q_\alpha}w(x,y).
$$
Since $w\geqslant \u{w}_{Q_\alpha}$ on $Q_\alpha\times Q_\alpha$ and the function $\vf_\alpha$ is bi-Lipschitz on $Q_\alpha$ with  constant $1+\ep$,  
$$\PP_s^w(Q_\alpha; N)\geqslant P^w_s(Q_{\alpha}; \omega_N)  \geqslant \u{w}_{Q_\alpha} P_s(Q_{\alpha}; \omega_N)\geqslant  (1+\ep)^{-s}\u{w}_{Q_\alpha} P_s(\tilde{Q}_\alpha; \tilde\omega_N),
$$
and thus,
\begin{equation}\label{PPsw1}
\PP_s^w(Q_\alpha;N) \geqslant (1+\ep)^{-s}\u{w}_{Q_\alpha}\PP_s(\tilde{Q}_\alpha; N).
\end{equation}
Dividing both sides of \eqref{PPsw1} by $\tau_{s,d}(N)$ and  then taking the  limit inferior as $N\to \infty$ gives
\begin{multline}\label{notreallysmall}
\u{h}_{s,d}^w(Q_\alpha)\geqslant (1+\ep)^{-s}\u{w}_{Q_\alpha} \u{h}_{s,d}(\tilde{Q}_\alpha)\geqslant (1+\ep)^{-s}\u{w}_{Q_\alpha} \sigma_{s,d} \L_d(\tilde{Q}_\alpha)^{-s/d}\\\geqslant  (1+\ep)^{-2s}\sigma_{s,d}\u{w}_{Q_\alpha} \H_d(Q_\alpha)^{-s/d},
\end{multline}
where  the second inequality follows from Lemma \ref{lemmafromabove}.

Finally, we apply Lemma \ref{addit} to $A=\clos(A_{\rm bi}^c)\cup \clos(D) \cup \left(\mathsmaller{\bigcup}_{\alpha} Q_{\alpha}\right)$. Combining \eqref{reallysmallone}, \eqref{reallysmalltwo} and \eqref{notreallysmall}, we obtain
\begin{multline}\label{yetanotherestimate}
[\u{h}_{s,d}^w(A)]^{-d/s}\leqslant [\u{h}_{s,d}^w(\clos(A_{\rm bi}^c))]^{-d/s}+[\u{h}_{s,d}^w(\clos(D))]^{-d/s}+\sli_{\alpha}[\u{h}_{s,d}^w(Q_\alpha)]^{-d/s} \\ \leqslant C_{s,d}^{-d/s}\H^{s,w}_d(\clos(D)) + (1+\ep)^{2d}\sigma_{s,d}^{-d/s}\sli_\alpha \u{w}_{Q_\alpha}^{-d/s} \H_d(Q_\alpha).
\end{multline}

Define  
$$
\u{w}_\ep(x):=\begin{cases} \u{w}_{Q_\alpha}^{-d/s}, & x\in Q_\alpha\ {\rm for \ some}\ \alpha,\\
0, &x\not\in \cup_\alpha Q_\alpha. \end{cases}
$$
Then \eqref{yetanotherestimate} implies
\begin{equation}\label{fromaboveforunderh}
[\u{h}_{s,d}^w(A)]^{-d/s} \leqslant  C_{s,d}^{-d/s}\H^{s,w}_d(\clos(D)) + (1+\ep)^{2d}\sigma_{s,d}^{-d/s}\int_A \u{w}_\ep(x){\rm d}\H_d(x).
\end{equation}
Observe that $\mathcal H_d(\partial_A Q_\alpha)=0$ for every $\alpha$ and that the set $A\setminus \left(\cup _{\alpha} {\rm int}_A Q_{\alpha}\right)$ is closed, where ${\rm int}_A Q_{\alpha}$ is the interior of $Q_\alpha$ relative to $A$. Recall also that $Q_\alpha\subset B$ for all $\alpha$ and that the sets $Q_\alpha$ are pairwise disjoint. Then
$$
D\subset {\rm clos}(D)\subset  D\cup \clos (A_{\rm bi}^c)\cup \left(\cup_{\alpha}\partial _A Q_{\alpha}\right).
$$
Consequently, $\mathcal H_d(\clos (D))=\mathcal H_d(D)<\gamma=\epsilon$. Then $\mathcal H_d^{s,w}(\clos (D))\to 0$ as $\epsilon\to 0$. Since ${\rm diam} (Q_{\alpha})\leqslant 2\epsilon(1+\epsilon)$ for all $\alpha$, for every $\epsilon>0$ sufficiently small, we have $Q_{\alpha}\times Q_{\alpha}\subset G$ for every $\alpha$, where the set $G$ is a neighborhood of $D(A)$ relative to $A\times A$ such that $a:=\inf _G w>0$, see Definition \ref {CPDweight}. This implies for sufficiently small $\ep>0$
\begin{equation}\label{boundeddd}
0\leqslant \u{w}_\ep(x)\leqslant a^{-d/s}.
\end{equation}
For every $k\in \NN$ denote $\ep_k:=2^{-k}$ and $\{Q_\alpha^k\}:=\mathcal{X}_k:=\mathcal{X}_{\ep_k, \ep_k}$. Let 
$$
M:=\left\{x\in B\colon \exists \ep_0 \; \mbox{such that} \; \forall \ep_k\leqslant\ep_0 \; \mbox{we have}\; x\in Q_\alpha^k \; \mbox{for some}\; Q_\alpha^k\in \mathcal{X}_k \right\}.
$$
We see that
$$
B\setminus M = \bigcap_{\ep_0} \bigcup_{k\geqslant \log_2(1/\ep_0)} (B\setminus \cup_\alpha Q_\alpha^k),
$$
thus for any $\ep_0>0$ we have
$$
\H_d(B\setminus M) \leqslant \sli_{k\geqslant  \log_2(1/\ep_0)} 2^{-k} \leqslant 2\ep_0,
$$
which implies $\H_d(B\setminus M)=0$. On the other hand, it is obvious that for every $x\in M$ we have $\lim_{k\to \infty} \u{w}_{\ep_k}(x) = w^{-d/s}(x,x)$. Using the estimate \eqref{fromaboveforunderh} for $\ep_k$ and in view of \eqref{boundeddd} and the Lebesgue Dominated Convergence Theorem, we obtain \eqref {frombeloww}. 
\end{proof}

\section{Limit distribution of asymptotically optimal configurations}\label {S8}

In this section we prove that   asymptotically $(s,w)$-optimal sequences of $N$-point configurations   are  distributed on the set $A$ according to $\mathcal{H}_d^{s,w}$.

Throughout this section,   $A$ will denote a set in $\R^p$ that satisfies the hypotheses of Theorem \ref{mainmain} (including    $\H_d(A)>0$)
and  $\{\omega_N\}_{N\geqslant 1}$ will denote an asymptotically $(s,w)$-optimal sequence of configurations  (see Definition~\ref{asympconf}) in $A$.

We start with the following lemma.
\begin{lemma}\label{manypoints} Let $\epsilon$ and $\gamma$ be positive numbers and  $\mathcal{X}_{\ep, \gamma}$ be as in Lemma \ref{vitali}.
Let $\tilde{Q}_\alpha=\vf_\alpha(Q_\alpha)$ for some fixed $Q_\alpha\in \mathcal{X}_{\ep, \gamma}$. 
Suppose  $\tilde\Gamma$    is a $d$-dimensional open cube contained in $\tilde{Q}_{\alpha}$ and let $\Gamma:=\vf_\alpha^{-1}(\tilde{\Gamma})$ and $N_\Gamma:=\#(\omega_N \cap \Gamma)$ for $N\in \NN$. Then $N_\Gamma \to \infty$ as $N\to \infty$.
\end{lemma}
\begin{proof}
Suppose there is an unbounded set $\mathcal{N}$ of positive integer numbers such that $N_\Gamma$ are uniformly bounded from above when $N\in \mathcal{N}$. Since $\vf_\alpha$ is a bi-Lipschitz function, there is a positive number $a_0$ (that does not depend on $N$) and, for each $N\in \mathcal{N}$,  a point $z_N\in A$  such that $B(z_N, a_0)\cap A \subset \Gamma$ and $B(z_N, a_0)\cap \omega_N=\emptyset$. Therefore, $|z_N-x|\geqslant a_0$ for any $x\in \omega_N$. Recall that we denote $D(A)=\{(x,x)\colon x\in A\}$. Since the set $F:=\textup{clos}(\cup_{N\in \mathcal{N}}\{(z_N, x)\colon x\in \omega_N\})$ is a closed subset of $A\times A$ with $F\cap D(A)=\emptyset$, we conclude from Definition \ref{CPDweight} that the weight $w$ is bounded from above on $F$. Then for some constant $C$ and any large enough $N\in \mathcal{N}$,
$$
P_s^w(A; \omega_N) \leqslant U^w(z_N;  \omega_N)   \leqslant C\cdot N.
$$
Since $\{\omega_N\}_{N\geqslant 1}$ is asymptotically $(s,w)$-optimal, we have $\u{h}_{s,d}^w(A) = 0$, which contradicts the fact that $\u {h}_{s,d}^w(A)>0$ established in Lemma~\ref{weightedfrombelow}. Then $N_\Gamma\to \infty$ as $N\to\infty$.
\end{proof}
The next lemma makes the asymptotic behavior of $N_\Gamma$ more precise.
\begin{lemma}\label{distr}
 Let $\epsilon$, $\Gamma$  and $N_\Gamma$ be as above. Then
\begin{equation}\label{distrrrinf}
\liminf_{N\to \infty}\frac{\tau_{s,d}(N_\Gamma)}{\tau_{s,d}(N)} \geqslant  \frac{\H_d(\Gamma)^{s/d}\u{h}^w_{s,d}(A)}{\sigma_{s,d}(1+\ep)^{2s}\o{w}_\Gamma}
\end{equation}
and
\begin{equation}\label{limsupfrombelow}
\limsup_{N\to \infty}\frac{\tau_{s,d}(N_\Gamma)}{\tau_{s,d}(N)} \geqslant \frac{\H_d(\Gamma)^{s/d}\o{h}^w_{s,d}(A)}{\sigma_{s,d}(1+\ep)^{2s}\o{w}_\Gamma},
\end{equation}
where
$$\
\o{w}_\Gamma := \sup_{(y,x)\in \Gamma\times \Gamma}w(y,x).
$$
\end{lemma}
\begin{proof}
Let  the sidelength of $\tilde{\Gamma}$ be denoted by $r>0$. For $0<\upsilon<r$, let $\tilde{\Gamma}_\upsilon$   denote the closed $d$-dimensional cube with the same center as $\tilde{\Gamma}$ and sidelength $r-\upsilon$. Denote $\Gamma_\upsilon:=\vf_\alpha^{-1}(\tilde{\Gamma}_\upsilon)$.

For any $N\geqslant 1$, 
$$
P_s^w(A; \omega_N) \leqslant P_s^w( \Gamma_\upsilon ;  \omega_N) = \inf_{y\in \Gamma_\upsilon} \left( \sli_{x \in \omega_N\cap \Gamma}\frac{w(y,x)}{|y-x|^s} + \sli_{x \in \omega_N\setminus \Gamma}\frac{w(y,x)}{|y-x|^s}\right).
$$
If $y\in \Gamma_\upsilon$ and $x\in Q_\alpha\setminus \Gamma$ then $|\vf_\alpha(y)-\vf_\alpha(x)| \geqslant \upsilon/2$, thus $|y-x|\geqslant {(1+\ep)^{-1}\upsilon/2}$. Furthermore, $h:={\rm dist}(\Gamma_{\upsilon},A\setminus Q_\alpha)>0$ since $\Gamma_{\upsilon}$ is a compact subset of the interior of $Q_\alpha$. Then for any $y\in \Gamma_{\upsilon}$ and $x\in A\setminus \Gamma=(A\setminus Q_\alpha)\cup (Q_\alpha\setminus \Gamma)$, we have $\left|y-x\right|\geqslant \min\{h,(1+\ep)^{-1}\upsilon/2\}>0$. This means that the set $F_1:=\textup{clos}(\Gamma_{\upsilon}\times (A\setminus \Gamma))\subset A\times A$ does not intersect the diagonal $D(A)$. Thus, the weight $w$ is bounded above on $F_1$ by a constant (which can depend on $\upsilon$). Consequently,
\begin{equation}
\label{9.2a}
P_s^w(A; \omega_N) \leqslant P_s^w( \Gamma_\upsilon ;    \omega_N\cap \Gamma)  + C_{\upsilon, \ep}\cdot N \leqslant \o{w}_\Gamma \cdot P_s( \Gamma_\upsilon ;    \omega_N\cap \Gamma)+ C_{\upsilon, \ep}\cdot N,
\end{equation}
where $C_{\upsilon,\epsilon}$ is a constant independent on $N$ and $\omega_N$.
 Let $\tilde \omega_N^\Gamma:=\varphi_\alpha(\omega_N\cap \Gamma)\subset \tilde \Gamma$.  Since $\varphi_\alpha$ is bi-Lipschitz with constant $(1+\epsilon)$, we have using \eqref{9.2a} that
$$
P_s^w(A; \omega_N) \leqslant (1+\ep)^{s} \o{w}_\Gamma P_s( \tilde{\Gamma}_\upsilon; \tilde{\omega}_N^\Gamma) + C_{\gamma, \ep}\cdot N.
$$

For any $\tilde{x}\in \tilde{\omega}_N^\Gamma$, define $\tilde{x}'$ to be the point in $\tilde\Gamma_\upsilon$ closest to $\tilde{x}$ (in particular, $\tilde{x}'=\tilde{x}$ if $\tilde{x}\in\tilde{\Gamma}_\upsilon$). Denote $\tilde{\omega}_N':=\{\tilde{x}'\colon \tilde{x}\in \tilde{\omega}_N^\Gamma\}$. Notice that $\#\tilde{\omega}_N' = N_\Gamma$. Since $\tilde\Gamma_\upsilon$ is a convex set, for any $\tilde y\in \tilde\Gamma_{\upsilon}$ we have $|\tilde y-\tilde{x}|\geqslant |\tilde y-\tilde{x}'|$. Thus,
\begin{equation}\label{trampampam}
\begin{split}
P_s^w(A; \omega_N) &\leqslant (1+\ep)^{s} \o{w}_\Gamma P_s( \tilde{\Gamma}_\upsilon; \tilde{\omega}_N')+ C_{\upsilon, \ep}\cdot N  \\
&\leqslant 
 (1+\ep)^{s}\o{w}_\Gamma \PP_s(\tilde{\Gamma}_\upsilon; N_\Gamma) + C_{\upsilon, \ep}\cdot N \\ 
 &=(1+\ep)^{s}\o{w}_\Gamma \H_d(\tilde\Gamma_\upsilon)^{-s/d} \PP_s(\mathcal{Q}_d, N_\Gamma) + C_{\upsilon, \ep}\cdot N.
\end{split}
\end{equation}
We now divide by $\tau_{s,d}(N)$ and take the limit inferior as $N\to \infty$. Using Lemma \ref {manypoints} and \eqref{cube222}, we obtain
$$
\u{h}_{s,d}^w(A) \leqslant (1+\ep)^{s} \o{w}_\Gamma (r-\upsilon)^{-s} \sigma_{s,d} \cdot \liminf_{N\to \infty}\frac{\tau_{s,d}(N_\Gamma)}{\tau_{s,d}(N)}.
$$
Since the number $\upsilon$ can be arbitrarily small, the function $\vf_\alpha$ is bi-Lipschitz, and $\mathcal H_d(\tilde \Gamma)=r^d$, we further obtain
$$
\u{h}_{s,d}^w(A)\leqslant (1+\ep)^{2s}\o{w}_\Gamma \H_d(\Gamma)^{-s/d} \sigma_{s,d}\cdot \liminf_{N\to \infty}\frac{\tau_{s,d}(N_\Gamma)}{\tau_{s,d}(N)},
$$
which proves \eqref{distrrrinf}.
Similarly, passing to $\limsup_{N\to \infty}$ in \eqref{trampampam}, we obtain
$$
\o{h}_{s,d}^w(A)\leqslant (1+\ep)^{2s}\o{w}_\Gamma \H_d(\Gamma)^{-s/d} \sigma_{s,d}\cdot \limsup_{N\to \infty}\frac{\tau_{s,d}(N_\Gamma)}{\tau_{s,d}(N)},
$$
which proves \eqref{limsupfrombelow}.
\end{proof}
Finally, we state the main lemma of this section, which proves the limiting behavior~\eqref{maindistr}.
\begin{lemma}\label{lemmadistr}
Suppose $B\subset A$ is a set with $\H_d(\partial_A B)=0$. Suppose $\{\omega_N\}_{N\geqslant 1}$ is an asymptotically $(s,w)$-optimal sequence of configurations in $A$. Then
$$
\lim_{N\to \infty}\frac{\#(\omega_N \cap B)}{N}=\frac{\H^{s,w}_d(B)}{\H^{s,w}_d(A)}.
$$
Hence, 
$$\label{maindistr2}
\nu(\omega_N)\stackrel{*}{\longrightarrow} \frac{1}{\H^{s,w}_d(A)} \H^{s,w}_d  \quad  \text{ as } \quad N\to \infty.
$$
\end{lemma}
\begin{proof}
If $\H_d(B)=0$ then clearly
$$
\liminf_{N\to \infty}\frac{\#(\omega_N \cap B)}{N}\geqslant\frac{\H^{s,w}_d(B)}{\H^{s,w}_d(A)}.
$$
Therefore, it remains to prove this inequality for $B$ with $\H_d(B)>0$.
Denote $B_{\rm bi}:={\textup{int}_A (B\setminus \clos(A_{\rm bi}^c))}$, where ${\rm int}_{A}\ \! X$ denotes the interior of a set $X\subset A$ relative to $A$.
For an $\ep>0$ consider the family $\mathcal{X}_{\ep, \ep}=\{Q_\alpha\}$ from Lemma \ref{vitali} constructed for the set $B_{\rm bi}$. Then $B_{\rm bi}=(\cup_\alpha Q_\alpha )\cup D$ with $\H_d(D)<\ep$. For each $\tilde{Q}_\alpha:=\vf_\alpha(Q_\alpha)$, consider a finite family $\tilde{\mathfrak{G}}_{\alpha}$ of disjoint open cubes $\tilde{\Gamma}\subset \tilde{Q}_\alpha$ (the families $\mathfrak G_\alpha$ will be specified later). Denote $\mathfrak{G}_\alpha:=\{\vf_\alpha^{-1}(\tilde{\Gamma})\colon \tilde\Gamma\in\tilde{\mathfrak{G}}_\alpha\}$ and let $\mathfrak{G}:=\bigcup_{\alpha}\mathfrak{G}_\alpha$.
Recall that for any $\Gamma \in \mathfrak{G}$ we define $N_\Gamma:=\#(\omega_N \cap \Gamma)$.

Notice that if $s\geqslant d$ then $\tau_{s,d}(N_\Gamma)/\tau_{s,d}(N)\leqslant(N_\Gamma/N)^{s/d}$ (in the case $s>d$ we have equality, while if $s=d$ we use $\log N_\Gamma \leqslant \log N$).  Then Lemma \ref{distr} implies that for every $\Gamma=\varphi_\alpha^{-1}(\tilde \Gamma)\in \mathfrak {G}$, we have
\begin{equation}\label{distrbigger}
\liminf_{N\to \infty} \frac{N_\Gamma}{N} \geqslant (1+\ep)^{-2d} \o{w}_\Gamma^{-d/s}\left(\frac{\u{h}_{s,d}^w(A)}{\sigma_{s,d}}\right)^{d/s} \cdot \H_d(\Gamma).
\end{equation}
 Since all sets $\Gamma\in \mathfrak{G}$ are disjoint, from \eqref {distrbigger} we have
\begin{equation}\label{distrrrrrr}
\begin{split}
\liminf_{N\to\infty} \frac{\#(\omega_N\cap B)}{N}&\geqslant \liminf\limits_{N\to\infty}\frac {\#(\omega_N\cap (\cup_\alpha Q_\alpha))}{N} \geqslant \liminf_{N\to \infty} \frac{1}N\sli_{\Gamma\in \mathfrak G} N_\Gamma  \geqslant \sli_{\Gamma\in \mathfrak G} \liminf_{N\to\infty} \frac{N_\Gamma}N\\ & \geqslant  (1+\ep)^{-2d} \left(\frac{\u{h}_{s,d}^w(A)}{\sigma_{s,d}}\right)^{d/s} \cdot \sli_{\Gamma\in \mathfrak G}\o{w}_\Gamma^{-d/s}\H_d(\Gamma).
\end{split}
\end{equation}

Fix a positive number $\upsilon$. Since $\L_d(\partial \tilde{Q}_\alpha)=0$ for every $\alpha$, we can choose the family $\tilde{\mathfrak{G}}_\alpha$ such that
\begin {equation}\label {43a}
\L_d\left(\tilde{Q}_\alpha\setminus \bigcup_{\tilde{\Gamma}\in \tilde{\mathfrak{G}}_\alpha}\tilde{\Gamma}\right)<\upsilon,
\end {equation}
and denote 
$$
\tilde{G}:=\bigcup_{\tilde{\Gamma}\in \tilde{\mathfrak{G}}}\tilde{\Gamma}, \;\;\;\; G:=\bigcup_{\Gamma\in \mathfrak{G}}\Gamma.
$$

Since the family $\{Q_\alpha\}$ is finite, for some constant $C_{\ep}$, which does not depend on $\upsilon$, 
\begin{equation}
\H_d(\cup_\alpha Q_\alpha \setminus G)\leqslant C_{\ep}\cdot \upsilon.
\end{equation}

Notice that $\tilde{G}$ is a finite union of open cubes. If we subdivide these cubes into smaller ones and call their union $\tilde{G}_1$, then $\L_d(\tilde{G}_1)=\L_d(\tilde G)$ and, moreover, the estimate \eqref{distrrrrrr} holds for the new collection $\tilde{\mathfrak{G}}_1$. We repeat this procedure and denote by $\tilde{\mathfrak{G}}_n$ the collection we get on the $n$'th step; we further denote by $\mathfrak G_n$ the collection of preimages of cubes from $\mathfrak{G}_n$. Then the maximum of the diameters of   cubes in $\tilde{\mathfrak{G}}_n$, and thus of every set in $\mathfrak{G}_n$ approaches 0 as $n\to \infty$; thus, as in the proof of Lemma~\ref{weightedfrombelow}, the Lebesgue Dominated Convergence Theorem applied to $u_n(x)=\sum_{\Gamma\in \mathfrak{G}_n}\o{w}_\Gamma^{-d/s}\chi_{\Gamma}(x)$ (where $\chi_{\Gamma}$ denotes the characteristic function of $\Gamma$) implies
$$
\sli_{\Gamma\in \mathfrak G_n}\o{w}_\Gamma^{-d/s}\H_d(\Gamma)\to \ili_{G}w^{-d/s}(x,x)\textup{d}\H_d(x), \;\; n\to\infty.
$$
Since $w^{-d/s}(x,x)$ is bounded away from zero, and $\H_d(\cup_\alpha Q_\alpha\setminus G)\leqslant C_\ep\cdot \upsilon$ for $\upsilon$ arbitrary small, we obtain from \eqref{distrrrrrr}
\begin{equation}
\liminf_{N\to\infty} \frac{\#(\omega_N\cap B)}{N}\geqslant (1+\ep)^{-2d} \left(\frac{\u{h}_{s,d}^w(A)}{\sigma_{s,d}}\right)^{d/s} \cdot \ili_{\cup_\alpha Q_\alpha}w^{-d/s}(x,x)\textup{d}\H_d(x).
\end{equation}

Finally, since $\H_d(B\setminus \cup_\alpha Q_\alpha)=\H_d(B_{\textup{bi}}\setminus \cup_\alpha Q_\alpha)<\ep$ and $\ep$ can be made arbitrarily small, we obtain using Lemma \ref{weightedfrombelow} that
\begin{equation}\label{distrliminf}
\liminf_{N\to\infty} \frac{\#(\omega_N\cap B)}{N}\geqslant \left(\frac{\u{h}_{s,d}^w(A)}{\sigma_{s,d}}\right)^{d/s} \cdot
\H^{s,w}_d(B) =\frac{\H^{s,w}_d(B)}{\H^{s,w}_d(A)} .
\end{equation}

Notice that a similar estimate is true for the set $A\setminus B$. Thus,
\begin{equation}\label{distrlimsup}
\limsup_{N\to \infty}  \frac{\#(\omega_N\cap B)}{N} = 1-\liminf_{N\to \infty} \frac{\#(\omega_N\cap (A\setminus B))}{N} \leqslant 1-
 \frac{\H^{s,w}_d(A\setminus B)}{\H^{s,w}_d(A)} = \frac{\H^{s,w}_d(B)}{\H^{s,w}_d(A)}.
\end{equation}
Combining estimates \eqref{distrliminf} and \eqref{distrlimsup}, we obtain
$$
\lim_{N\to \infty}  \frac{\#(\omega_N\cap B)}{N} = \frac{\H^{s,w}_d(B)}{\H^{s,w}_d(A)}.
$$
\end{proof}

\section{An estimate for $\o{h}^w_{s,d}$ from above}\label {S9}

In this section we prove that the lower bound for $\u{h}_{s,d}^w(A)$ from Lemma~\ref{weightedfrombelow} is also an upper bound for $\o{h}_{s,d}^w(A)$. In view of Lemmas \ref{weightedfrombelow} and \ref{lemmadistr}, this completes the proof of Theorem~\ref{mainmain}.
\begin{lemma}
Suppose $A\subset \R^p$ is a compact set with $\H_d(A)=\mathcal{M}_d(A)<\infty$ and that $\H_d(\clos(A_{\rm bi}^c))=0$. Suppose $w$ is a CPD-weight on $A\times A$ with parameter $d$. Then for any $s\geqslant d$, we have\begin{equation}\label{frombelow}
\o{h}_{s,d}^w(A)\leqslant \frac{\sigma_{s,d}}{\H^{s,w}_d(A)^{s/d}}.
\end{equation}
\end{lemma}
\begin{proof}
If $\H_d(A)=0$, then inequality \eqref {frombelow} holds trivially. Assume that $\H_d(A)>0$. Set $B:=A\setminus \clos(A_{\rm bi}^c)$. Then $B$ is a relatively open subset of $A_{\rm bi}$. For a positive number $\ep>0$, fix the family $\mathcal{X}_{\ep, \ep}$ from Lemma \ref{vitali}. Let $\{\omega_N\}_{N\geqslant 1}$ be an asymptotically optimal sequence of configurations for $\PP^w_s(A; N)$.
Let $\Gamma\subset B$ be a set as in Lemma \ref{manypoints}.
Recall the estimate in \eqref{limsupfrombelow}:
$$
\limsup_{N\to \infty}\left(\frac{N_\Gamma}{N}\right)^{s/d}\geqslant \limsup_{N\to\infty}\frac {\tau_{s,d}(N_\Gamma)}{\tau_{s,d}(N)}\geqslant (1+\ep)^{-2s}\o{w}_\Gamma^{-1}\frac{\o{h}^w_{s,d}(A)}{\sigma_{s,d}}\H_d(\Gamma)^{s/d}.
$$
Since $\H_d(\partial_A \Gamma) = 0$, Lemma \ref{lemmadistr} implies that the limit $\lim\limits_{N\to \infty}\frac{N_\Gamma}{N}$ exists. Then
$$
(1+\ep)^{-2d}\left(\frac{\o{h}^w_{s,d}(A)}{\sigma_{s,d}}\right)^{d/s}\o{w}_\Gamma^{-d/s}\H_d(\Gamma) \leqslant \lim_{N\to \infty}\frac{N_\Gamma}{N}.
$$
We now argue exactly as in Lemma \ref{lemmadistr}. That is, we take the sequence of families $\{\mathfrak{G}_n\}_{n=0}^\infty$ from the proof of Lemma \ref {lemmadistr} and obtain
$$
(1+\ep)^{-2d}\left(\frac{\o{h}^w_{s,d}(A)}{\sigma_{s,d}}\right)^{d/s}\sli_{\Gamma\in\mathfrak{G}_n}\o{w}_\Gamma^{-d/s}\H_d(\Gamma) \leqslant \sli_{\Gamma\in \mathfrak{G}_n}\lim_{N\to \infty}\frac{N_\Gamma}{N} = \lim_{N\to \infty}\frac{1}N\sli_{\Gamma\in \mathfrak{G}_n}N_\Gamma \leqslant 1.
$$

Passing to the limit as $n\to\infty$ we obtain
$$
(1+\ep)^{-2d}\left(\frac{\o{h}^w_{s,d}(A)}{\sigma_{s,d}}\right)^{d/s}\H_d^{s,w}({\rm clos}(B))\leqslant 1,
$$
which in view of $\H_d(\clos(A_{bi}^c))=0$ implies

$$
\left(\frac{\o{h}^w_{s,d}(A)}{\sigma_{s,d}}\right)^{d/s} \H^{s,w}_d(A) \leqslant 1,
$$
which completes the proof of \eqref{frombelow}.
\end{proof}
\bibliography{ref}

\begin{thebibliography}{10}

\bibitem{Ambrus2009}
G.~Ambrus.
\newblock {A}nalytic and {P}robabilistic {P}roblems in {D}iscrete {G}eometry.
\newblock {\em Thesis (Ph.D.)--University College London}, 2009.

\bibitem{Ambrus2013}
G.~Ambrus, K.~M. Ball, and T.~Erd{\'e}lyi.
\newblock {C}hebyshev constants for the unit circle.
\newblock {\em Bull. Lond. Math. Soc.}, 45(2):236--248, 2013.

\bibitem{Borodachov2014}
S.~V. Borodachov and N.~Bosuwan.
\newblock {A}symptotics of discrete {R}iesz {$d$}-polarization on subsets of
  {$d$}-dimensional manifolds.
\newblock {\em Potential Anal.}, 41(1):35--49, 2014.

\bibitem{Borodachov2007}
S.~V. Borodachov, D.~P. Hardin, and E.~B. Saff.
\newblock {A}symptotics of best-packing on rectifiable sets.
\newblock {\em Proc. Amer. Math. Soc.}, 135(8):2369--2380, 2007.

\bibitem{Borodachov2008}
S.~V. Borodachov, D.~P. Hardin, and E.~B. Saff.
\newblock {A}symptotics for discrete weighted minimal {R}iesz energy problems
  on rectifiable sets.
\newblock {\em Trans. Amer. Math. Soc.}, 360(3):1559--1580, 2008.

\bibitem{Borodachov2015}
S.~V. Borodachov, D.~P. Hardin, and E.~B. Saff.
\newblock {\em {M}inimal {D}iscrete {E}nergy on {R}ectifiable {S}ets}.
\newblock Springer, 2016.

\bibitem{MR1662447}
J.~H. Conway and N.~J.~A. Sloane.
\newblock {\em Sphere packings, lattices and groups}, volume 290 of {\em
  Grundlehren der Mathematischen Wissenschaften [Fundamental Principles of
  Mathematical Sciences]}.
\newblock Springer-Verlag, New York, third edition, 1999.
\newblock With additional contributions by E. Bannai, R. E. Borcherds, J.
  Leech, S. P. Norton, A. M. Odlyzko, R. A. Parker, L. Queen and B. B. Venkov.

\bibitem{Dahlberg1978}
Bj{\"o}rn Dahlberg.
\newblock {O}n the distribution of {F}ekete points.
\newblock {\em Duke Mathematical Journal}, 45(3):537--542, 1978.

\bibitem{Erdelyi2013}
T.~Erd{\'e}lyi and E.~B. Saff.
\newblock {R}iesz polarization inequalities in higher dimensions.
\newblock {\em J. Approx. Theory}, 171:128--147, 2013.

\bibitem{Farkas2008}
B.~Farkas and B.~Nagy.
\newblock {T}ransfinite diameter, {C}hebyshev constant and energy on locally
  compact spaces.
\newblock {\em Potential Anal.}, 28:241--260, 2008.

\bibitem{Farkas2006}
B.~Farkas and S.~G. R{\'e}v{\'e}sz.
\newblock {P}otential theoretic approach to rendezvous numbers.
\newblock {\em Monatsh. Math.}, 148(4):309--331, 2006.

\bibitem{Farkas2006a}
B.~Farkas and S.~G. R{\'e}v{\'e}sz.
\newblock {R}endezvous numbers of metric spaces---a potential theoretic
  approach.
\newblock {\em Arch. Math. (Basel)}, 86(3):268--281, 2006.

\bibitem{Hardin2013}
D.~P. Hardin, A.~P. Kendall, and E.~B. Saff.
\newblock {P}olarization optimality of equally spaced points on the circle for
  discrete potentials.
\newblock {\em Discrete Comput. Geom.}, 50(1):236--243, 2013.

\bibitem{Hardin2005}
D.~P. Hardin and E.~B. Saff.
\newblock {M}inimal {R}iesz energy point configurations for rectifiable
  {$d$}-dimensional manifolds.
\newblock {\em Adv. Math.}, 193(1):174--204, 2005.

\bibitem{Kuijlaars1998}
A.~B.~J. Kuijlaars and E.~B. Saff.
\newblock {A}symptotics for minimal discrete energy on the sphere.
\newblock {\em Trans. Amer. Math. Soc.}, 350(2):523--538, 1998.

\bibitem{Landkof1972}
Naum~Samo{\u\i}lovich Landkof.
\newblock {\em {F}oundations of modern potential theory}.
\newblock Springer-Verlag, New York-Heidelberg, 1972.

\bibitem{Mattila1995}
P.~Mattila.
\newblock {\em {G}eometry of sets and measures in {E}uclidean spaces. Fractals
  and rectifiability}, volume~44 of {\em Cambridge Studies in Advanced
  Mathematics}.
\newblock Cambridge University Press, Cambridge, 1995.

\bibitem{Ohtsuka1967}
M.~Ohtsuka.
\newblock {O}n various definitions of capacity and related notions.
\newblock {\em Nagoya Math. J.}, 30:121--127, 1967.

\bibitem{reznikov2016minimum}
A.~Reznikov, E.B. Saff, and O.V. Vlasiuk.
\newblock A minimum principle for potentials with application to chebyshev
  constants.
\newblock {\em arXiv preprint arXiv:1607.07283}, 2016.

\bibitem{MR1485778}
E.~B. Saff and V.~Totik.
\newblock {\em Logarithmic potentials with external fields}, volume 316 of {\em
  Grundlehren der Mathematischen Wissenschaften [Fundamental Principles of
  Mathematical Sciences]}.
\newblock Springer-Verlag, Berlin, 1997.
\newblock Appendix B by Thomas Bloom.

\bibitem{Simanek2015}
B.~Simanek.
\newblock {A}symptotically optimal configurations for chebyshev constants with
  an integrable kernel.
\newblock {\em New York Journal of Mathematics}, 22:667--675, 2016.

\bibitem{Su2014}
Y.~Su.
\newblock {\em Discrete {M}inimal {E}nergy on {F}lat {T}ori and {F}our-{P}oint
  {M}aximal {P}olarization on {$\mathbb{S}^2$}}.
\newblock ProQuest LLC, Ann Arbor, MI, 2015.
\newblock Thesis (Ph.D.)--Vanderbilt University.

\bibitem{thomson1904xxiv}
J.~J. Thomson.
\newblock On the structure of the atom: an investigation of the stability and
  periods of oscillation of a number of corpuscles arranged at equal intervals
  around the circumference of a circle; with application of the results to the
  theory of atomic structure.
\newblock {\em The London, Edinburgh, and Dublin Philosophical Magazine and
  Journal of Science}, 7(39):237--265, 1904.

\end{thebibliography}
\bibliographystyle{plain}

\end{document}